\documentclass{amsart}[11pt]

\def\volume{\operatorname{vol}}

\def\op{\operatorname}
\def\svolball#1#2{{\volume(\underline B_{#2}^{#1})}}
\def\svolann#1#2{{\volume(\underline A_{#2}^{#1})}}

\def\svolsp#1#2{{\volume(\partial \underline B_{#2}^{#1})}}

\usepackage[usenames,dvipsnames]{color}
\usepackage{amsmath,amssymb,bm,amsthm, amsfonts, mathrsfs, eucal, epsfig}
\usepackage{graphicx}
\usepackage{url}
\allowdisplaybreaks
\usepackage[top=2.5cm,bottom=2.0cm,left=2.5cm,right=2.5cm]{geometry}

\makeatletter 
\@addtoreset{equation}{section}
\makeatother  

\begin{document}

\newtheorem{Thm}{Theorem}[section]
\newtheorem{Def}{Definition}[section]
\newtheorem{Lem}[Thm]{Lemma}
\newtheorem{Rem}{Remark}[section]

\newtheorem{Cor}[Thm]{Corollary}
\newtheorem{sublemma}{Sub-Lemma}
\newtheorem{Prop}{Proposition}[section]
\newtheorem{Example}{Example}[section]
\newcommand{\g}[0]{\textmd{g}}
\newcommand{\pr}[0]{\partial_r}
\newcommand{\dif}{\mathrm{d}}
\newcommand{\bg}{\bar{\gamma}}
\newcommand{\md}{\rm{md}}
\newcommand{\cn}{\rm{cn}}
\newcommand{\sn}{\rm{sn}}
\newcommand{\seg}{\mathrm{seg}}

\newcommand{\Ric}{\mbox{Ric}}
\newcommand{\Iso}{\mbox{Iso}}
\newcommand{\ra}{\rightarrow}
\newcommand{\Hess}{\mathrm{Hess}}
\newcommand{\RCD}{\mathsf{RCD}}

\title{Almost maximal volume entropy rigidity for integral Ricci curvature in the non-collapsing case}
\author{Lina Chen}
\address[Lina Chen]{Department of mathematics, Nanjing University, Nanjing China}

\email{chenlina\_mail@163.com}
\thanks{Chen partially supported by the NSFC 12001268 and a research fund from Nanjing University.} 

\maketitle

\begin{abstract}

\setlength{\parindent}{10pt} \setlength{\parskip}{1.5ex plus 0.5ex
minus 0.2ex} 
In this note we will show the almost maximal volume entropy rigidity for manifolds with lower integral Ricci curvature bound in the non-collapsing case:
Given $n, d, p>\frac{n}{2}$, there exist $\delta(n, d,  p), \epsilon(n, d,  p)>0$, such that for $\delta<\delta(n, d,  p)$, $\epsilon<\epsilon(n, d, p)$, if a compact $n$-manifold $M$ satisfies that the integral Ricci curvature has lower bound $\bar k(-1, p)\leq \delta$, the diameter $\op{diam}(M)\leq d$ and volume entropy $h(M)\geq n-1-\epsilon$, then the universal cover of $M$ is Gromov-Hausdorff close to a hyperbolic space form $\Bbb H^k$, $k\leq n$; If in addition the volume of $M$, $\volume(M)\geq v>0$, then $M$ is diffeomorphic and Gromov-Hausdorff close to a hyperbolic manifold where $\delta, \epsilon$ also depends on $v$.
  \end{abstract}

 \section{Introduction}
 
 A Riemannian $n$-manifold $M$  has lower integral Ricci curvature bound if there are constants $p>\frac{n}{2}, R>0, H$ such that
 $$\bar k(H ,p, R)=\sup_{x\in M}\left(\frac1{\volume(B_R(x))}\int_{B_R(x)} \rho_H^p dv\right)^{\frac{1}{p}} =\sup_{x\in M}\left(-\kern-1em\int_{B_R(x)} \rho_H^p dv\right)^{\frac{1}{p}}$$
 has an upper bound, where $\rho_H=\max\{-\rho(x)+(n-1)H, 0\}$ and $\rho\left( x\right) $ is the smallest
eigenvalue for the Ricci tensor $\op{Ric} : T_xM\to T_xM$. If $R=\op{diam}(M)$, the diameter of $M$, we also denote $\bar k(H, p, R)=\bar k(H, p)$.
Compared with manifolds with lower Ricci curvature bound, many basic properties and results have been generalized to manifolds with lower integral Ricci curvature bound. For instance, the  Laplacian comparison \cite{PW1, Au1}, relative volume comparison \cite{PW1,CW} and almost splitting theorem, almost metric cone rigidity \cite{PW2, TZ, Ch2}. 

In this note, we will study the almost maximal volume entropy rigidity for compact manifolds with lower integral Ricci curvature bound.

Recall that for a compact manifold $M$, the volume entropy of $M$ is defined as 
$$h(M)=\lim_{R\to \infty}\frac{\ln\volume(B_R(\tilde x))}{R},$$
where $\tilde x\in \tilde M$, the universal cover of $M$ and the limit always exists and independent of $\tilde x$ \cite{Ma}.

For a compact $n$-manifold $M$ with lower Ricci curvature bound $\op{Ric}_M\geq -(n-1)$, the maximal volume entropy rigidity says that  $h(M)\leq n-1$ (by Bishop volume comparison) and when $h(M)$ obtains the maximal $h(M)=n-1$, $M$ is isometric to a hyperbolic manifold \cite{LW}. And the almost maximal volume entropy rigidity says that for $\op{diam}(M)\leq d$,  when $h(M)$ is close to $n-1$, $M$ is diffeomorphic and Gromov-Hausdorff close to a hyperbolic manifold \cite{CRX}. For non-smooth metric measure spaces with some curvature dimension condition ($\RCD$ or Alexandrov), there are also some rigidity results about volume entropy taking maximal or almost maximal \cite{CDNPSW, Ji, Ch1}. 

For compact manifolds $M$ with lower integral Ricci curvature bound, $\bar k(-1, p)\leq \delta<\delta_0(n, p, d)$,  in \cite{CW}, with G. Wei, we gave an optimal estimate of volume entropy, 
$$h(M)\leq n-1+c(n, p,d)\delta^{\frac12}, \quad d=\op{diam}(M).$$ 

In this note we will show that for integral Ricci curvature when $h(M)$ is close to $n-1$, the universal cover $\tilde M$ is Gromov-Hausdorff close to a hyperbolic space form $\Bbb H^k$, $k\leq n$:
\begin{Thm}
Given $n, d, p>\frac{n}{2}$, there exist $\delta(n, d,  p), \epsilon(n, d,  p)>0$, such that for $0< \delta<\delta(n, d,  p)$, $0<\epsilon<\epsilon(n, d, p)$, if a compact $n$-manifold $M$ satisfies that
$$\op{diam}(M)\leq d, \quad \bar k(-1, p)\leq \delta, \quad h(M)\geq n-1-\epsilon,$$
then $\tilde M$ is $\Psi(\delta, \epsilon | n, d, p)$-Gromov-Hausdorff close to a simply connected hyperbolic space form $\Bbb H^k$, $k\leq n$, where $\Psi\to 0$ when $\delta, \epsilon\to 0$ and $n, d, p$ fixed.
\end{Thm}

By this result, it is easy to derive that in the non-collapsing case, i.e., $\volume(M)\geq v>0$, the almost maximal volume entropy rigidity holds:

 \begin{Thm}
Given $n, d, v>0, p>\frac{n}{2}$, there exist $\delta(n, d, v, p), \epsilon(n, d, v, p)>0$, such that for $0< \delta<\delta(n, d, v, p)$, $0<\epsilon<\epsilon(n, d, v, p)$, if a compact $n$-manifold $M$ satisfies that
$$\op{diam}(M)\leq d, \quad \bar k(-1, p)\leq \delta, \quad \volume(M)\geq v,\quad h(M)\geq n-1-\epsilon,$$
then $M$ is diffeomorphic to a hyperbolic $n$-manifold by a $\Psi(\delta, \epsilon | n, d, v, p)$-isometry.
\end{Thm}

The proofs of our results is similar as the one of \cite[Theorem D]{CRX}. The key part is using the almost maximal volume entropy condition 
$$h(M)\geq n-1-\epsilon,$$
and relative volume comparison to derive that for each fixed $\rho>2d$, $B_{\rho}(\tilde x)\subset \tilde M$ is Gromov-Hausdorff close to a ball in a warped product space
$\Bbb R\times_{e^r}Y$ where $Y$ is a connected length space. Then the same argument as \cite[Lemma 4.4]{CRX} gives that $\tilde M$ is Gromov-Hausdorff close to the hyperbolic manifold $\Bbb H^k$.

A difficulty here in the key part is that we do not have effective relative volume comparison between two arbitrary large annuluses in the integral Ricci curvature case.  However we can use the almost maximal volume entropy condition to supply all the controls we need (see \eqref{alm-vol-1} and \eqref{alm-vol-2} and Claim1-3 of the proof of Theorem~\ref{warp-pro}). 

The author would like to thank Professor Guofang Wei for the recommendation of the topic of this note.

\section{Preliminaries}
In this section, we will give some properties we need for manifolds with integral Ricci curvature bound.
For a complete $n$-manifold $M$, $x\in M$,  let $r=d(x, \cdot)$ be the distance function from $x$ and let 
$\psi=\max\{\Delta r-\underline{\Delta}_H r, 0\}$, where $\underline{\Delta}_H$ is the Laplacian operator in the simply connected space $\underline{M}_H^n$ with constant sectional curvature $H$. Then $\psi=0$ if $\op{Ric}_M\geq (n-1)H$. 
\begin{Thm}  \label{com-int}
Given $n, p> \frac{n}{2}$,$R>0, H$, for a complete $n$-manifold $M$, fix $x \in M$, 
then the following holds:

{\rm (2.1.1)} Laplacian comparison estimates \cite[Lemma 2.2]{PW1}: 
\begin{equation}
-\kern-1em\int_{B_R(x)}\psi \leq  c(n, p) \bar k^{\frac12}(H, p, R), \quad  c(n, p)= \left(\frac{(n-1)(2p-1)}{2p-n}\right)^{\frac{1}{2}}; \label{lap-com}
\end{equation}

{\rm (2.1.2)} Relative volume comparison  \cite{PW1, CW}:
There exist $\delta_0=\delta(n, p, H), c=c(n, p, H)$, such that if $\bar k(H, p, 1)\leq \delta\leq \delta_0$, then for each $0<r<R$, $R<\frac{\pi}{\sqrt H}$ for $H>0$,
 \begin{equation}\frac{\volume(\partial B_R(x))}{\volume(B_R(x))}\leq \frac{\svolsp{H}{R}}{\svolball{H}{R}}+c\bar k^{\frac12}(H, p, R), \label{rel-vol-1}\end{equation}
 \begin{equation}
 \frac{\volume(A_{r, R}(x))}{\svolann{H}{r, R}}\leq \frac{\volume(B_r(x))}{\svolball{H}{r}}\left(1+\left(e^{(\max\{R, 1\}-r)c\delta^{\frac12}}-1\right)\frac{\svolball{H}{R}}{\svolann{H}{r, R}}\right), \label{rel-vol-2}
 \end{equation}
 \begin{equation}
 \frac{\volume(B_R(x))}{\svolball{H}{R}}\leq e^{(\max\{R, 1\}-r)c\delta^{\frac12}}\frac{\volume(B_r(x))}{\svolball{H}{r}}. \label{rel-vol-3}
 \end{equation}

\end{Thm}

If $M$ has integral Ricci curvature bound, then by \cite{Au2}, the universal cover  $\tilde M$ also has integral Ricci curvature bound:
\begin{Thm}\cite[Lemma 1.0.1]{Au2} \label{univ-cover}
Given $n, p>\frac{n}{2}, d>0, H$, there is $\delta=\delta(n, p, d, H)>0,\ c(n, H, d)>0$ such that if a compact Riemannian $n$-manifold $M$ satisfies that
$$\bar k(H, p)\leq \delta, \quad \op{diam}(M)\leq d,$$
then for any $\tilde x\in \tilde M$, the universal cover of $M$, and for any $R\geq 3d$, 
\begin{equation}
\left(-\kern-1em\int_{B_R(\tilde x)}\rho_H^p\right)^{\frac1p}\leq c(n, H, d)\, \bar k(H, p).  \label{univ-rho}
\end{equation}
\end{Thm}

By the relative volume comparison, the set of manifolds with integral Ricci curvature bound is precompact.
\begin{Thm}[\cite{PW1, Au2} Precompactness] \label{compact}
For $n\geq 2, p>\frac{n}{2}, H$, there exists $c(n, p, H)$ such that if a sequence of  compact Riemannian $n$-manifold $M_i$ satisfies that 
$\op{diam}(M)^2\bar{k_i}(H, p)\leq c(p, n,  H)$, then  there are subsequences of $\{(M_i, x_i)\}$ and $\{(\tilde M_i, \tilde x_i)\}$ that converge in the pointed Gromov-Hausdorff topology where $\tilde M_i$ is the universal cover of $M_i$.
\end{Thm}

\section{Proofs of the main results}

\subsection{Warped product structure}
In this subsection, we will prove the following almost warped product rigidity for manifolds with lower integral Ricci curvature (compared with \cite[Theorem 1.4]{CRX}).

\begin{Thm} \label{key-lem}
Given $p>\frac{n}2, d>0$, there are $\delta_0=\delta_0(n, p, d)$, $\epsilon_0=\epsilon_0(n, p, d)$, such that for $\delta<\delta_0, \epsilon<\epsilon_0$, if a compact Riemannian $n$-manifold $M$ satisfies that 
$$\bar k(-1, p)\leq \delta, \quad h(M)\geq n-1-\epsilon, \quad \op{diam}(M)\leq d,$$
then for each $D>8d$,
there are $L_i\to \infty$ large, such that for any $\rho\in (2d, {D}/{4})$, there are disjoint metric balls,
$B_{\rho}(q^i_j)\subset A_{L_i, L_i+D}(\tilde x)$, such that 
$$d_{GH}\left(B_{\rho}\left(q^i_j\right),
	B_{\rho}\left((0,y)\right)\right)\leq \Psi\left(\delta, \epsilon, L_i^{-1} | n, d, p, D\right),$$
	where $B_{\rho}((0,y))$ is a metric ball in a warped
	product space $\Bbb R^1\times_{e^s} Y$, $Y$ is a connected length space. And 
	\begin{equation}\frac{\volume(\cup_jB_{\rho}(q^i_j))}{\volume(A_{L_i, L_i+D}(\tilde x))} \geq c(n, p, d, D)>0. \label{max}\end{equation}
\end{Thm}

As the discussion in \cite{CRX}, by \cite[Theorem 3.6]{CC1} (see also \cite[Theorem 5.1]{CRX} where the condition $\op{Ric}_M\geq (n-1)H$ can be replaced by the condition that Segment inequality (see \cite{Ch2} for integral Ricci curvature) and local doubling property \eqref{rel-vol-3} hold), to prove Theorem~\ref{key-lem}, we only need to show that

\begin{Thm} \label{warp-pro}
	Let the assumptions be as in Theorem~\ref{key-lem}. For each $D>8d$, there are $L_i\to \infty$ large, such that for $\rho\in (2d, D/4)$, $0<\alpha<1$, there are disjoint metric balls,
	$B_{\rho}(q^i_j)\subset A_{L_i, L_i+D}(\tilde x)$, satisfying \eqref{max} and 
	
	\noindent (i) for $y\in B_\rho(q^i_j)$, there is $z\in \partial B_{L_i+ D}(\tilde x)$
	satisfying $|\tilde x y|+ |yz|\leq L_i+D+ \Psi(\delta, \epsilon, L_i^{-1} |n, p, d,D)$;
	
	\noindent (ii) for each $q^i_j$, let $u(y)=|y \tilde x|-|q^i_j \tilde x|$, there is a
	smooth function $\tilde f$ satisfying
	
	(ii1) $|\tilde f-e^u|<\Psi(\delta, \epsilon, L_i^{-1} |n, p, D,d)$ for all $x\in B_{(1-\alpha)\rho}(q^i_j)$.
	
	(ii2) $-\kern-1em\int_{B_\rho(q^i_j)}|\nabla \tilde f-\nabla e^u|^2\le
	\Psi(\delta, \epsilon, L_i^{-1} |n,p, D,d)$.
	
	(ii3) $-\kern-1em\int_{B_{(1-\alpha)\rho}(q^i_j)}|\op{Hess}\tilde f-e^u|^2
	\le \Psi(\delta, \epsilon, L_i^{-1}|n,p,D,\alpha,d)$.

\end{Thm}

Consider a compact $n$-manifold $M$ as in Theorem~\ref{key-lem}. By the definition of volume entropy, we know that $h(M)\geq n-1-\epsilon$ implies that there is $R_0>0$ such that  for $R>R_0$,
 \begin{equation}\frac{\volume(\partial B_R(\tilde x))}{\volume(B_R(\tilde x))}\geq n-1-\epsilon. \label{alm-vol-1}\end{equation}
 And by \cite[Lemma 4.2]{CRX}, for each $D>8d$, there are  $L_i\to \infty$ such that for $i>i_0$
 \begin{equation}
 \frac{\volume(\partial B_{L_i+D}(\tilde x))}{\volume(\partial B_{L_i}(\tilde x))}\geq e^{(n-1-\epsilon)D}-\epsilon. \label{alm-vol-2}
 \end{equation}
 
 By \eqref{univ-rho}, if $\bar k(-1, p)\leq \delta<\delta(n, p, d)$, for any $R>3d$, in the universal cover of $M$, $\tilde M$,
 \begin{equation}
 \bar k_{\tilde M}(-1, p, R)\leq c(n, p, d)\bar k(-1, p)\leq c(n, p, d)\delta. \label{uni-cur}
 \end{equation}

 \begin{proof}[Proof of Theorem~\ref{warp-pro}]
 
 Let $L_i\to \infty$, $i\geq i_0$, be as in \eqref{alm-vol-2}.  Take $\delta_0$ as in (2.1.2) and Theorem~\ref{univ-cover}.
 
 Claim 1: 
 $$ -\kern-1em\int_{A_{L_i, L_i+D}(\tilde x)} \left| \Delta r - (n-1)\right|\leq \Psi(\delta, \epsilon, L_i^{-1} | n, p, d, D).$$
 
  By \eqref{alm-vol-1}, we have that for  $L>R_0$ and $D>0$ fixed,
 \begin{eqnarray*}
 \volume(A_{L, L+D}(\tilde x))&= & \int_L^{L+D}\volume(\partial B_{R}(\tilde x)) dR\\
 &\geq & \int_L^{L+D}(n-1-\epsilon)\volume(B_R(\tilde x))dR\geq \int_L^{L+D}(n-1-\epsilon)\volume(B_L(\tilde x))dR \\
 &=& (n-1-\epsilon)D\volume(B_L(\tilde x)).
 \end{eqnarray*}
 Thus
 \begin{equation}
 \frac{\volume(B_{L+D}(\tilde x))}{\volume(A_{L, L+D}(\tilde x))}=1+\frac{\volume(B_L(\tilde x))}{\volume(A_{L, L+D}(\tilde x))}
 \leq  1+ \frac1{(n-1-\epsilon)D}. \label{abgeq}
 \end{equation}
  
 By Laplacian  comparison \eqref{lap-com}, \eqref{uni-cur} and \eqref{abgeq},
 \begin{eqnarray*}
 -\kern-1em\int_{A_{L, L+D}(\tilde x)}  \Delta r &\leq & -\kern-1em\int_{A_{L, L+D}(\tilde x)}  \underline{\Delta} r +\psi\\
 &= & -\kern-1em\int_{A_{L, L+D}(\tilde x)}  \underline{\Delta} r +\frac{\volume(B_{L+D}(\tilde x))}{\volume(A_{L, L+D}(\tilde x))} -\kern-1em\int_{B_{L+D}(\tilde x)}\psi\\
 &\leq &-\kern-1em\int_{A_{L, L+D}(\tilde x)}  \underline{\Delta} r +\left(1+\frac1{(n-1-\epsilon)D}\right) -\kern-1em\int_{B_{L+D}(\tilde x)}\psi\\
 & \leq & n-1+ \Psi(\delta, L^{-1} | n, p, d, D),
 \end{eqnarray*}
where in the last inequality we used
$$\underline{\Delta}r=(n-1)\frac{\cosh r}{\sinh r}\to n-1, \text{ as } r\to \infty.$$ 
 
 On the other hand,  for $L>\max\{L_{i_0}, R_0\}$, $L=L_i$, by \eqref{alm-vol-1}, \eqref{alm-vol-2}, \eqref{rel-vol-2}
 \begin{eqnarray*}
 -\kern-1em\int_{A_{L, L+D}(\tilde x)}  \Delta r &= &  \frac{\volume(\partial B_{L+D}(\tilde x))-\volume(\partial B_L(\tilde x))}{\volume(A_{L, L+D}(\tilde x))}\\
 & \geq & \left(e^{(n-1-\epsilon)D}-1-\epsilon\right)\frac{\volume(\partial B_L(\tilde x))}{\volume(A_{L, L+D}(\tilde x))}\\
  & = & \left(e^{(n-1-\epsilon)D}-1-\epsilon\right)\frac{\volume(\partial B_L(\tilde x))}{\volume(B_L(\tilde x))}\frac{\volume(B_L(\tilde x))}{\volume(A_{L, L+D}(\tilde x))}\\
  & \geq & \left(e^{(n-1-\epsilon)D}-1-\epsilon\right)(n-1-\epsilon)\left(1+\left(e^{Dc\delta^{\frac12}}-1\right)\frac{\svolball{-1}{L+D}}{\svolann{-1}{L, L+D}}\right)^{-1}\frac{\svolball{-1}{L}}{\svolann{-1}{L, L+D}}\\
  &\geq & n-1-\Psi(\delta, \epsilon, L^{-1} | n, p, d, D),
 \end{eqnarray*}
 where in the last inequality we used
 $$\lim_{L\to \infty}\frac{\svolball{-1}{L}}{\svolann{-1}{L, L+D}}=\frac1{e^{(n-1)D}-1}.$$
 
 As in \cite{CRX},
let $E$ be a maximal subset of $\{q_i, \,B_{\rho}(q_i)\subset A_{L, L+D}(\tilde x)\}$
such that for all $q_{i_1}\neq q_{i_2} \in E$, $B_{\rho}(q_{i_1})\cap B_{\rho}(q_{i_2})=\emptyset$.
Let $F=\bigcup_{q_i\in E}B_{\rho}(q_i)$.  

Claim 2:  There is $C(n, p, d, D)>0$ such  that
\begin{equation}\frac{\volume(F)}{\volume(A_{L, L+D}(\tilde x))}\geq C(n, p, d, D). \label{F-large}\end{equation}

Let $G=\bigcup_{q_i\in E}B_{2\rho}(q_i)$. 
By the maximality of $E$, $A_{L+\rho, L+D-\rho}(\tilde x) \subset G$.	
Then by \eqref{alm-vol-1}, \eqref{rel-vol-2}
\begin{eqnarray}
\frac{\volume(G)}{\volume(A_{L, L+D}(\tilde x))} &\geq & \frac{\volume(A_{L+\rho, L+D-\rho}(\tilde x))}{\volume(A_{L, L+D}(\tilde x))} \nonumber\\
& = & \frac{\int_{L+\rho}^{L+D-\rho}\volume(\partial B_R(\tilde x))dR}{\volume(A_{L, L+D}(\tilde x))} \nonumber \\
&\geq &\frac{(n-1-\epsilon)(D-2\rho)\volume(B_L(\tilde x))}{\volume(A_{L, L+D}(\tilde x))} \nonumber \\
&\geq & (n-1-\epsilon)(D-2\rho)\left(1+\left(e^{Dc\delta^{\frac12}}-1\right)\frac{\svolball{-1}{L+D}}{\svolann{-1}{L, L+D}}\right)^{-1}\frac{\svolball{-1}{L}}{\svolann{-1}{L, L+D}}. \nonumber \\
&\geq & C(n, p, d, D). \label{G-large}
\end{eqnarray}
By \eqref{rel-vol-3},
$$\frac{\volume(B_{2\rho}(x))}{\svolball{-1}{2\rho}}\leq e^{DC\delta^{1/2}_0} \frac{\volume(B_{\rho}(x))}{\svolball{-1}{\rho}}.$$
And thus
\begin{equation}\frac{\volume(F)}{\volume(G)}\geq \frac{\sum_{q_i\in E} \volume(B_{\rho}(q_i))}{\sum_{q_i\in E}
		\volume(B_{2\rho}(q_i))}\geq e^{-DC\delta_0^{1/2}}\frac{\svolball{-1}{\rho}}{\svolball{-1}{2\rho}}. \label{FG-large}\end{equation}
		
Then \eqref{F-large} is derived  by \eqref{G-large} and \eqref{FG-large}.

Let 
$$S=\{y\in B_{L+D}(\tilde x),\, \exists z\in \partial B_{L+D}(\tilde x), d(y, \tilde x)+d(z, y)=L+D\}.$$ 

Claim 3:
\begin{equation}
\frac{\volume(S\cap A_{L, L+D}(\tilde x))}{\volume(A_{L, L+D}(\tilde x))}\geq 1-\Psi(\epsilon, \delta, L^{-1} | n, p, D, d). \label{pass-thr}
\end{equation}

By \eqref{rel-vol-1}, \eqref{alm-vol-1}
\begin{eqnarray*}
\frac{\volume(S)}{\svolball{-1}{L+D}} & \geq & \left(1+c\delta^{\frac12}\frac{\svolball{-1}{L+D}}{\svolsp{-1}{L+D}}\right)^{-1}\frac{\volume(\partial B_{L+D}(\tilde x))}{\svolsp{-1}{L+D}}\\
& \geq & \left(1+c\delta^{\frac12}\frac{\svolball{-1}{L+D}}{\svolsp{-1}{L+D}}\right)^{-1}(n-1-\epsilon)\frac{\volume(B_{L+D}(\tilde x))}{\svolsp{-1}{L+D}},
\end{eqnarray*}
i.e.,
\begin{equation}
\frac{\volume(S)}{\volume(B_{L+D}(\tilde x))}\geq  \left(1+c\delta^{\frac12}\frac{\svolball{-1}{L+D}}{\svolsp{-1}{L+D}}\right)^{-1}(n-1-\epsilon)\frac{\svolball{-1}{L+D}}{\svolsp{-1}{L+D}}.
\end{equation}
Since 
$$\lim_{L\to \infty}\frac{\svolball{-1}{L+D}}{\svolsp{-1}{L+D}}=\frac1{n-1},$$
we know that for $L$ large, 
$$\frac{\volume(S)}{\volume(B_{L+D}(\tilde x))}\geq 1-\Psi(\delta, \epsilon, L^{-1} | n, D, p, d).$$

Note that by \eqref{alm-vol-1}, \eqref{rel-vol-3},
\begin{eqnarray*}
\frac{\volume(A_{L, L+D}(\tilde x))}{\volume(B_{L+D}(\tilde x))} & = &\frac{\int_L^{L+D}\volume(\partial B_R(\tilde x)) dR}{\volume(B_{L+D}(\tilde x))}\\
&\geq &\frac{\int_L^{L+D}(n-1-\epsilon)\volume(B_R(\tilde x)) dR}{\volume(B_{L+D}(\tilde x))}\\
&\geq & \frac{(n-1-\epsilon)D\volume(B_{L}(\tilde x))}{\volume(B_{L+D}(\tilde x))}\\
&\geq & (n-1-\epsilon)De^{-Dc\delta^{\frac12}}\frac{\svolball{-1}{L}}{\svolball{-1}{L+D}}.
\end{eqnarray*}
Thus \eqref{pass-thr} holds.

By Claim 2 and Claim 3, we have that 
\begin{equation}
\frac{\volume(S\cap F)}{\volume(F)}\geq 1-\Psi(\epsilon, \delta, L^{-1} | n, D, p,d).  \label{pass-thr-2}
\end{equation}

Take $0<\eta<1$ (which will be specified later) and let
$$E'(\eta)=\left\{q_i\in E,\,\,\, \;\frac{\volume(B_\rho(q_i)\setminus S)}{\volume(B_\rho(q_i))}<\eta\right\},$$
 $F'(\eta)=\bigcup_{q_i\in E'(\eta)}B_{\rho}(q_i)$.
 Then
\begin{equation}\frac{\volume(F'(\eta))}{\volume(F)}\ge1-\eta^{-1}\Psi(\delta,\epsilon, L^{-1}|n, D, p, d). \label{F1-large}\end{equation}

In fact, by \eqref{pass-thr-2}
$$\Psi(\epsilon, \delta, L^{-1} | n, D, p, d)\volume(F)\geq \volume(F\setminus S)\geq \volume((F\setminus F'(\eta))\setminus S)\geq \eta\volume(F\setminus F'(\eta)).$$

Now as the discussion of \cite[Lemma 5.7, 5.8]{CRX}, by Claim 1 and Claim 2, we can take $\eta=\Psi^{\frac12}(\delta, \epsilon, L^{-1} | n, p, D, d)$ such that for
$$E''(\eta)=\left\{q_i\in E,\,\, -\kern-1em\int_{B_\rho(q_i)} \left|\Delta r-(n-1)\right|<\eta^{-1}\Psi(\delta, \epsilon,L^{-1}|n,p,D, d)\right\},$$
and  $F''(\eta)=\bigcup_{q_i\in E''(\eta)}B_{\rho}(q_i)$ we have that
\begin{equation}\frac{\volume(F''(\eta))}{\volume (F)}\ge 1-\eta. \label{F2-large}\end{equation}

Take $q\in E'(\eta)\cap E''(\eta)$. Then by the definition of $E'(\eta)$ and doubling property \eqref{rel-vol-3},  (i) holds for any $y\in B_{\rho}(q)$ and 
\begin{equation}-\kern-1em\int_{B_\rho(q)} \left|\Delta r-(n-1)\right|<\Psi(\delta, \epsilon,L^{-1}|n,p,D, d).\label{lap-ball}\end{equation}

Let $f=e^u, u(y)=|\tilde xy|-|\tilde xq|$ and let  $\tilde f$ be the solution of
$$\begin{cases}\Delta \tilde f = ne^u, & \text{in } B_\rho(q);\\
\tilde f =f, &\text{on } \partial B_\rho(q).\end{cases}$$
Then as in \cite[Lemma 5.9]{CRX}, by \eqref{lap-ball}, we have that
\begin{eqnarray}
	-\kern-1em\int_{B_{\rho}(q)\setminus C_p}|\Delta (\tilde f-e^u)|&=&-\kern-1em\int_{B_{\rho}(q)\setminus C_p} |n e^u - e^u(|\nabla u|^2 + \Delta u)| \nonumber\\
	&=& -\kern-1em\int_{B_{\rho}(q)\setminus C_p}e^u|n-1-\Delta u| \nonumber\\
	&\leq & \Psi(\delta, \epsilon, L^{-1}|n,p,D ,d). \label{lap-f}
	\end{eqnarray}
By a standard argument as in \cite[Lemma 5.9]{CRX}, we have (ii1)-(ii3). Namely, Maximal principle \cite{DWZ} (cf. \cite[Theorem 2.4]{Ch2}) and \eqref{lap-f} implies $|\tilde f-f|\leq c(n, p, D, d)$. Then integral by parts gives (ii2). By Gradient estimates \cite{DWZ} (cf. \cite[Theorem 2.5]{Ch2}) and (ii2) we have that  $|\nabla f-\nabla\tilde f|\leq c(n, p, D, d)$ which implies (ii1) by Segment inequality \cite{Ch2}. Finally, Bochner's formula, cut-off function \cite{DWZ} (cf. \cite[Lemma 2.8]{Ch2}) and (ii1), (ii2) gives (ii3).

And \eqref{max} follows by \eqref{F1-large}, \eqref{F2-large} and Claim 2.
\end{proof}

\subsection{Proofs of Theorem 1.1 and 1.2}

By Theorem~\ref{key-lem}, the following proofs of Theorem 1.1 and 1.2 are the same as in \cite{CRX}. Here we will give a rough discussion.

Argue by contradictions.  By the procompactness Theorem~\ref{compact}, assume there is a sequence of $n$-manifolds $M_i$ satisfies:
 \begin{equation}
 \op{diam}(M_i)\leq d, \quad \bar k_i(-1, p)\leq \delta_i, \quad h(M_i)\geq n-1-\epsilon_i
  \end{equation}
 with $\delta_i\to 0$, $\epsilon_i\to 0$ and the following communicate diagram:
 
 \begin{equation}\begin{array}[c]{ccc}
(\tilde M_i,\tilde x_i,\Gamma_i)&\xrightarrow{GH}&(\tilde X,\tilde x,G)\\
\downarrow\scriptstyle{\pi_i}&&\downarrow\scriptstyle{\pi}\\
(M_i,x_i)&\xrightarrow{GH} &(X, x),
\end{array} \end{equation}
 where $\Gamma_i=\pi_1(M_i, x_i)$ the fundamental group of $M_i$. Note that by \cite{Ket} and \cite{Ch2}, $\tilde X$ is a $\RCD(-(n-1), n)$-space. And by \cite{Ch2} or the theory of $\RCD$-spaces (see \cite{GR, So}), $G$ is a Lie group.
 
 To prove Theorem 1.1, we will show that $\tilde X$ is isometric to a simply connected hyperbolic space form $\Bbb H^k$, $k\leq n$. And to get Theorem 1.2, by \cite{PW2}, we only need to show that  if in addition there is $v>0$ such that $\volume(M_i)\geq v$, then $X$ is isometric to a hyperbolic $n$-manifold.

By Theorem~\ref{key-lem}, for any fixed $D>4\rho>8d$, there are $L_j\to \infty$ large, $B_{\rho}(q_j)\subset A_{L_j, L_j+D}(\tilde x_i)$, 
$$d_{GH}\left(B_{\rho}\left(q_j\right),
	B_{\rho}\left((0,y_i)\right)\right)\leq \Psi\left(\delta, \epsilon, L_j^{-1} | n, d, p, D\right),$$ 
where $B_{\rho}((0,y_i))$ is a metric ball in a warped product space $\Bbb R^1\times_{e^s} Y_i$, $Y_i$ is a connected length space.

Since $\op{diam}(M_i)\leq d$, there is $\gamma_{ij}\in \Gamma_i$ such that $B_{\frac{\rho}{2}}(\tilde x_i)\subset B_{\rho}(\gamma_{ij}(q_j))=\gamma_{ij}(B_{\rho}(q_j))$.  Thus
$$d_{GH}\left(B_{\frac{\rho}2}\left(\tilde x_i\right),
	B_{\frac{\rho}2}\left((0,y_i)\right)\right)\leq \Psi\left(\delta, \epsilon, L_j^{-1} | n, d, p, D\right).$$ 
Let $L_j\to \infty$, we have that 	
$$d_{GH}\left(B_{\frac{\rho}2}\left(\tilde{x_i}\right),
	B_{\frac{\rho}2}\left((0,y_i)\right)\right)\leq \Psi\left(\delta, \epsilon | n, d, p, D\right).$$
		
Let $i\to \infty$, we derive that for any  fixed $D>4\rho>8d$ there is a connected length metric space $Y$
	such that
	$$B_{\frac{\rho}2}\left(\tilde{x}\right)= B_{\frac{\rho}2}\left((0,y)\right),$$
	where $B_{\frac{\rho}2}((0,y))$ is a metric ball in a warped
	product space $\Bbb R^1\times_{e^s} Y$, i.e., 
$$\tilde X=\Bbb R\times_{e^s}Y.$$

Now by the regularity of $\tilde X$ (see \cite[Theorem 4.7]{Ch2} or use the regularity of $\RCD$-spaces), we can take $\tilde x$ as a regular point. Then the same argument as the proof of \cite[Lemma 4.4]{CRX} gives that $\tilde X=\Bbb H^k$, $k\leq n$.

If in addition there is $v>0$, such that $\volume(M_i)\geq v$, then $k=n$ and thus the identity component of $G$, $G_0=\{e\}$, i.e., $G$ is discrete. To see $X$ is a hyperbolic manifold, one only need to show $G$ acting freely on $\tilde X$. This fact can be seen by the same argument as in \cite[Theorem 2.1]{CRX} where one needs the volume convergence theorem \cite{PW2} and almost metric cone rigidity and almost splitting theorem  \cite{PW2, TZ, Ch2}.

\end{document}